\documentclass[final]{amsart}

\usepackage{tikz,tikz-cd} 
\usetikzlibrary{shapes.geometric,positioning}
\usetikzlibrary{arrows,decorations.pathmorphing,decorations.pathreplacing}
\usetikzlibrary{matrix,arrows.meta}
\usepackage{tkz-graph}
 \usepackage[all]{xy}
 \usepackage[colorlinks=true,pagebackref,hyperindex]{hyperref}
\usepackage[font=small]{caption}
\usepackage{fancyvrb}
\usepackage[normalem]{ulem}
\usepackage{xspace}

\usepackage{tikz}
\usepackage{tkz-graph}
\usetikzlibrary{arrows}
\usetikzlibrary{backgrounds}
\usetikzlibrary{calc}
\tikzstyle{bl}=[circle, draw=white, thin,fill=black!100, scale=0.5]
\tikzstyle{cg}=[circle, draw, thin,fill=black!10, scale=0.8]
\tikzstyle{cw}=[circle, draw, thin,fill=white, scale=0.8]

\usepackage{booktabs}
\usepackage{multirow}
\usepackage{colortbl}
\usepackage{boldline}
\usepackage{cellspace}
\newcommand{\txn}[1]{\textnormal{#1}}
\newcommand{\rsphere}{\mathbb{C}\mathbb{P}^1}
\newcommand{\algbr}{\overline{\mathbb{Q}}}
\newcommand{\absgal}{\txn{Gal}(\algbr/\mathbb{Q})}
\newcommand{\inv}[1]{#1^{-1}}

\newcommand{\per}{(\sigma,\alpha,\varphi)}

\newcommand{\belyi}{Bely\u{\i}\xspace}

\newcommand{\zoi}{\{0,1,\infty\}}
\usepackage{color}
\definecolor{blue}{rgb}{0.0, 0.0, 1.0}
\definecolor{candyapplered}{rgb}{1.0, 0.03, 0.0}

\usepackage{graphicx}
\usepackage[all]{xy}
\usepackage{placeins}
\usepackage{enumitem}
\usepackage{amssymb}
\usepackage{latexsym}
\usepackage{amsmath}
\usepackage{mathrsfs}
\usepackage{array,booktabs}
\usepackage{verbatim}
\usepackage{color}
\usepackage[normalem]{ulem}
\usepackage{float}

\usepackage{soul}

\usepackage{color}
\definecolor{teal}{rgb}{0.0, 0.5, 0.5}
\definecolor{tealblue}{rgb}{0.21, 0.46, 0.53}
\definecolor{tealgreen}{rgb}{0.0, 0.51, 0.5}
\definecolor{tuscanred}{rgb}{0.51, 0.21, 0.21}
\definecolor{sangria}{rgb}{0.57, 0.0, 0.04}
\definecolor{rufous}{rgb}{0.66, 0.11, 0.03}
\definecolor{pinegreen}{rgb}{0.0, 0.47, 0.44}
\definecolor{darkscarlet}{rgb}{0.34, 0.01, 0.1}
\definecolor{darkseagreen}{rgb}{0.56, 0.74, 0.56}
\definecolor{darkpastelred}{rgb}{0.76, 0.23, 0.13}
\definecolor{darkpink}{rgb}{0.91, 0.33, 0.5}
\definecolor{darkpastelblue}{rgb}{0.47, 0.62, 0.8}
\definecolor{alizarin}{rgb}{0.82, 0.1, 0.26}
\definecolor{candyapplered}{rgb}{1.0, 0.03, 0.0}

\usepackage{hyperref}

\newcommand{\hyref}[2]{ \hyperref[#2]{#1~\ref*{#2}} }

\theoremstyle{plain}
\newtheorem{theorem}{Theorem}[section]

\newtheorem{corollary}[theorem]{Corollary}
\newtheorem{proposition}[theorem]{Proposition}

\theoremstyle{definition}

\newtheorem{example}[theorem]{Example}
\newtheorem{definition}[theorem]{Definition}

\DeclareMathAlphabet{\mathpzc}{OT1}{pzc}{m}{it}

\renewcommand{\mod}[1]{\mathsf{mod}(#1)}

\renewcommand{\emptyset}{\varnothing}

\newcommand{\arr}{\ar@{-}[r]}

\renewcommand{\phi}{\varphi}
\renewcommand{\epsilon}{\varepsilon}

\begin{document}

\date{\today}

\parskip7pt
\parindent0pt

\hyphenation{Gro-then-di-eck Pro-gra-mme Ma-the-mat-i-cal en-fan-ts}

\title{Dessins d'enfants, Brauer graph algebras and Galois invariants }
\thanks{Most of this work has been carried out while the second author was working at the University of Leicester and it has been supported through the EPSRC  Early Career Fellowship EP/P016294/1.  The first author thanks the University of Leicester for their hospitality. The second author  would  also like to thank the Isaac Newton Institute for Mathematical Sciences, Cambridge, for support and hospitality during the programme 'Cluster Algebras and Representation Theory' where part of the work on this paper was undertaken. This work was supported by EPSRC grant EP/R014604/1.}

\author{Goran Mali\'c}
\address{Department of Mathematics, Smith College, Northampton, MA 01063, USA}
\email{goranm00@gmail.com}
\author{Sibylle Schroll} 
\address{Department of Mathematics, Universit\"at zu K\"oln, Weyertal 86-90, 50931 K\"oln, Germany}
\email{schroll@math.uni-koeln.de}

\begin{abstract}
In this paper, we associate a finite dimensional algebra, called a Brauer graph algebra,  to every clean dessin d'enfant by constructing a quiver based on the monodromy of the dessin. We show that  Galois conjugate dessins d'enfants give rise to derived equivalent Brauer graph algebras and that   the stable Auslander-Reiten quiver and the dimension of the Brauer graph algebra are invariant under the induced action of the absolute Galois group.
\end{abstract}

\maketitle


\section{Introduction}

The aim of this paper is to introduce a new algebraic structure  to a dessin d'enfant. This new algebraic structure corresponds to a well-known class of finite dimensional algebras arising from the representation theory of finite groups, called Brauer graph algebras.  We show that  the action of the absolute Galois group 
$\rm{Gal}(\overline{\mathbb{Q}}/\mathbb{Q})$ on dessins d'enfants naturally induces an action of $\rm{Gal}(\overline{\mathbb{Q}}/\mathbb{Q})$ on  Brauer graph 
algebras.
 It is our hope that the rich representation theory of Brauer graph algebras will be a source of new Galois invariants. That this might be the case is supported by the results in this paper. In particular, we show  that the dimension of the Brauer graph algebra associated to a clean dessin d'enfant is a Galois invariant and that two Galois conjugate clean dessins d'enfants give rise to derived equivalent Brauer graph algebras. This shows that the Galois orbits of Brauer graph algebras stratify the derived equivalence classes of Brauer graph algebras. 

A dessin d'enfant (dessin for short) is a connected bipartite graph cellularly embedded in a connected, closed and orientable surface, with vertices coloured in black and white. Dessins d'enfants were introduced by Alexandre Grothendieck in his Esquisse d'un Programme \cite{Grothendieck} as a means of studying $\absgal$  as the automorphism group of a certain topological object sewn together from moduli spaces of curves with marked points.  However, dessins d'enfants  can be traced as far back as 1879 where the idea already appears in a  paper  by Klein \cite{Klein}. The central result in the theory of dessins d'enfants is \belyi's theorem \cite{belyi80} which establishes a 1-1 correspondence between isomorphism classes of dessins and isomorphism classes of smooth projective curves defined over $\algbr$ equipped with a holomorphic projection to $\rsphere$ ramified over a subset of $\zoi$. A remarkable consequence of \belyi's theorem is that $\absgal$ acts faithfully on the set of dessins which enables us to \emph{see} the action of $\absgal$ on number fields as an action on embedded graphs. A major open problem of the theory is to understand the invariants of this action and consequently to be able to distinguish between any two orbits. Several invariants that distinguish between certain orbits are known, but we are still far from a satisfactory answer.

In this paper we will work with dessins whose white vertices always have degree 2. Such a dessin is called clean. We can identify a clean dessin $D$ with a (not necessarily bipartite) graph $G$ by ``forgetting'' the white vertices of $D$.

Brauer graph algebras originate in the modular representation theory of finite groups \cite{Janusz} and their representation theory is well understood. They coincide with the class of symmetric special biserial algebras \cite{Roggenkamp, Schroll}, and they are of  tame representation type \cite{WW}, that is, the isomorphism classes of their indecomposable representations can be parametrised by finitely many one parameter families. Brauer graph algebras are defined by their so-called Brauer graph, a vertex decorated graph with a cyclic ordering of the edges incident to each vertex. Remarkably, much of the representation theory of Brauer graph algebras is encoded in the Brauer graph. Examples of this are projective resolutions (of maximal uniserial submodules of projective indecomposables) which are encoded by the so-called Green walks \cite{Roggenkamp}, the structure of their module category which is encoded by the Auslander-Reiten quiver \cite{Erdmann-Skowronski, Duffield}, their wall and chamber structure in the form of the complex of 2-term tilting complexes \cite{Adachi-Aihara-Chan}, and the Lie algebra structure of their first Hochschild cohomology group \cite{Chaparro-Schroll-Solotar}.  

The idea of this paper arises from the fact that a Brauer graph (minus the vertex decorations) can be naturally seen as a cellularly embedded graph on a connected closed oriented surface where the orientation of the surface gives rise to the cyclic orderings of the edges at each vertex. see, for example, \cite{MS} or  \cite{schrollbrauer}.  This results in a 1-1 correspondence between isomorphism classes of clean dessins d'enfants and isomorphism classes of Brauer graph algebras without vertex decoration.

Our first result is to show that the Galois action preserves the bipartiteness of a clean dessin, where bipartiteness refers to the bipartiteness of the graph obtained from a clean dessin by ``forgetting'' the white vertices. 

\textbf{Theorem A. }[Theorem~\ref{lemma:bipartite}] \emph{
Let $D$ be a clean dessin. Then $D$ is (not) bipartite if and only if any Galois conjugate of $D$ is (not) bipartite. }

Combining Theorem A  with known results on Galois conjugate dessins, which we recall in Theorem 2.2,  and a recent result on derived equivalences of Brauer graph algebras \cite[Theorem 7.11]{Opp-Zvn}, we obtain that two Galois conjugate dessins give rise to derived equivalent Brauer graph algebras. 

\textbf{Theorem B. }[Corollary~\ref{corollary:derivedeq}]  \emph{Let $D_1$ and $D_2$ be Galois conjugate clean dessins d'enfants with Brauer graph algebras $\Lambda_1$ and $\Lambda_2$ defined over an algebraically closed field $K$}. Then $\Lambda_1$ and $\Lambda_2$ are derived equivalent.

The next result shows that Galois conjugation of dessins d'enfants is more refined than derived equivalences classes of the  corresponding Brauer graph algebras. Namely, we show that Galois conjugation  induces a stratification of the derived equivalence classes of Brauer graph algebras in terms of Galois orbits by showing that  the dimension of a Brauer graph algebra is Galois invariant. Namely, we show the following.

\textbf{Theorem C. }[Theorem~\ref{theorem:dimension}] \emph{Let $D_1$ and $D_2$ be Galois conjugate clean dessins with associated Brauer graph algebras $\Lambda_1$ and $\Lambda_2$. Then $\dim_K\Lambda_{1} =  \dim_K\Lambda_{2} $.
}

In our last result, we show  that the module category of a Brauer graph algebra is Galois invariant in the sense that the stable Auslander-Reiten quiver of a Brauer graph algebra is invariant under the action of the absolute Galois group in the following way.

\textbf{Theorem D. }[Proposition~\ref{proposition:excptubes}] \emph{Let $D$ be a dessin d'enfant and $A_D$ the associated Brauer graph algebra and suppose that $K$ is algebraically closed. Then the number and the rank of the exceptional tubes in the stable Auslander-Reiten quiver of $A_D$ are Galois invariants.
}

In conclusion, while we are not able to derive new Galois invariants from the  connection of dessins d'enfants and Brauer graph algebras, it is remarkable that so much of the representation theoretic structure of the Brauer graph algebras is Galois invariant. This potentially opens up a novel way to construct new Galois invariants.

\subsection*{Acknowledgements} We thank the anonymous referee for carefully reading the manuscript. Their comments and suggestions have improved the quality and clarity of the paper.

\section{Dessins d'enfants and Brauer graph algebras}

Let $G$ be a connected bipartite graph (possibly with multiple edges) with vertices coloured in black and white and let $X$ be a connected, closed and orientable surface. A dessin d'enfant is a cellular embedding $\iota\colon G\hookrightarrow X$ such that the vertices of $G$ are points or 0-cells on $X$, the edges are $1$-cells on $X$ which intersect only at the vertices, and the complement of $\iota(G)$ in $X$ is a disjoint union of open 2-cells called \emph{faces}. The segments on $X$ connecting black and white vertices are called \emph{edges}. Two dessins $G_1\hookrightarrow X_1$ and $G_2\hookrightarrow X_2$ are isomorphic if there is an orientation preserving homeomorphism $h\colon X_1\to X_2$ that restricts to a graph isomorphism $G_1\to G_2$ which sends the black, resp.\ white vertices of $G_1$ to the black, resp.\ white vertices of $G_2$. 

Equivalently, a dessin d'enfant is a pair $(X,f)$ where $X$ is a compact Riemann surface and $f\colon X\to\rsphere$ is a holomorphic ramified covering of $\rsphere$, ramified over a subset of $\zoi$. The preimage $\inv f([0,1])$ of the closed unit interval recovers the embedding of $G$ into $X$ where the preimages $\inv f(0)$ and $\inv f(1)$ correspond to the black and white vertices, respectively. The pair $(X,f)$ is called a \belyi pair and $f$ is called a \belyi function. Two dessins $(X_1,f_1)$ and $(X_2,f_2)$ are isomorphic if there is an orientation preserving homeomorphism $h\colon X_1\to X_2$ such that $f_1=f_2\circ h$.

\subsection{Grothendieck's correspondence and Galois action}

By \belyi's theorem \cite{belyi80,belyi02} a compact Riemann surface $X$, understood as a projective smooth algebraic curve, is defined over $\algbr$ if and only if there exists a \belyi pair $(X,f)$. Grothendieck noted in \cite{Grothendieck} that \belyi's theorem implies a 1-to-1 correspondence of isomorphism classes of dessins d'enfants understood as embedded graphs and the isomorphism classes of \belyi pairs $(X,f)$, and as a consequence there is a natural faithful action of the absolute Galois group $\absgal$ over the rationals on the family of all isomorphism classes of dessins \cite{JonesWolfart}. An automorphism $\theta\in\absgal$ acts on a dessin by acting on its \belyi pair $(X,f)$, i.e.\ by acting on the coefficients of $X$ and $f$. We will denote the action of $\theta\in\absgal$ on a dessin $D=(X,f)$ as $D^\theta=(X,f)^\theta=(X^\theta,f^\theta)$.

This correspondence is known as Grothendieck's correspondence and we will make use of it in Lemma \ref{lemma:bipartite}.

\begin{theorem}[Grothendieck's correspondence]\label{gro-cor} The collection of isomorphism classes of dessins d'enfants is in  1-1 correspondence with the collection of isomorphism classes of pairs $(X,f)$, where $X$ is a compact Riemann surface defined over $\algbr$ and $f\colon X\to\rsphere$ is a holomorphic ramified covering of $\rsphere$, ramified over a subset of $\zoi$.\end{theorem}

A major open problem of the theory of dessins d'enfants is to understand the invariants of the action of $\absgal$ on the collection of dessins, which we call Galois invariants, and consequently to be able to distinguish between any two orbits. The combinatorial nature of dessins results in a number of Galois invariants that can be immediately read from the drawing as a dessin, some of which we now summarize in the following Theorem.

\begin{theorem}[Galois invariants of dessins \cite{GGD,JonesWolfart,LandoZvonkin,schneps_lochak97vol1}]\label{theorem:galinv} Let $D$ be a dessin and $D^\theta$ its  conjugate by $\theta\in\absgal$. Then $D$ and $D^\theta$ have the same
	\begin{itemize}
		\item number and degree sequence of black vertices,
		\item number and degree sequence of white vertices,
		\item number of edges,
		\item number and degree sequence of faces,
		\item genus,
	\end{itemize}
where the degree of a vertex is the number of edges incident to it, the degree of a face is half the number of edges bounding it, and the genus is the topological genus of the underlying surface.
\end{theorem}

An interested reader may find further details on Grothendieck's correspondence, the action of $\absgal$ on the collection of dessins and more intricate Galois invariants in \cite{GGD,JonesWolfart,LandoZvonkin,schneps_lochak97vol1}.

\subsection{Clean dessins d'enfants}\label{section:clean} In this paper we will consider \emph{clean} dessins only, that is the dessins in which all white vertices are of degree 2. Equivalently these are the dessins such that the ramification degree of any point in $\inv f(1)$ is equal to 2. Recall that a face of a clean dessin is a connected region enclosed by a subset of edges, i.e.\ a connected component of the complement of the union of the vertices and the edges. Note that for dessins on the sphere, the ``outer'' region is also a face. Clean dessins are typically drawn with the white vertices omitted, see Figure \ref{figure:clean}.
\begin{figure}[ht]
	\centering
	\includegraphics[width=.3\textwidth]{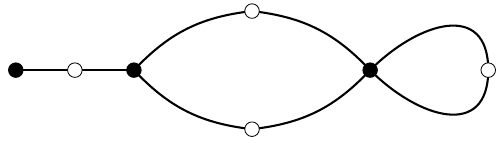}\qquad\includegraphics[width=.3\textwidth]{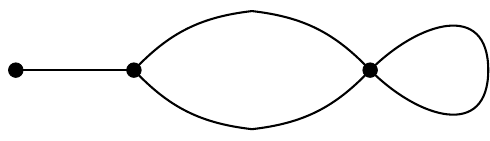}
	\caption{An example of a clean dessin with 3 black vertices, 4 white vertices, 3 faces, 8 edges.  On the right side the dessin is displayed with its white vertices deleted which is often how clean dessins are presented.}
	\label{figure:clean}
\end{figure}

Any dessin $(X,f)$ can be transformed into a clean dessin by replacing $f$ with $g\circ f$, where $g\colon\rsphere\to\rsphere$ is defined by $g(z)=4z(1-z)$ for $z\in\mathbb C$ and $\infty\mapsto\infty$ \cite{GGD}. The dessin $(X,g\circ f)$ is obtained from $(X,f)$ by colouring all the white vertices black and adjoining new white vertices between any two black vertices. Consequently, no algebraic curve over $\algbr$ is excluded by considering clean dessins only.

It is an obvious consequence of Theorem \ref{theorem:galinv} that any Galois conjugate of a clean dessin is again a clean dessin with the same number of vertices, edges and faces.

\subsection{Brauer graph algebras associated to dessins}\label{section:bga}

A quiver $Q = (Q_0, Q_1)$ is given by a finite set of vertices $Q_0$ and a finite set of oriented edges $Q_1$, called arrows, and functions $s,t : Q_1 \to Q_0$ where for $a \in Q_1$, $s(a)$ denotes the start of $a$ and $t(a)$ denotes its target.

Let $D$ be a clean dessin. We define the quiver $Q_D=(Q_0,Q_1)$ associated to $D$ as follows: the set $Q_0$ of vertices of $Q$ corresponds to the white vertices of $D$, and the arrows in $Q_1$ are given by the counter-clockwise orderings of edges around black vertices of degree at least 2, as well as pairs of consecutive (in the counter-clockwise order) parallel edges.

More precisely, let $b$ be a black vertex of degree $k\geq2$ and $w_1, \ldots, w_k$ its adjacent white vertices ordered counter-clockwise around $b$. Then $w_1, \ldots, w_k$ are in $Q_0$ and we add an arrow $w\to w'$ to $Q_1$ for every pair of consecutive white vertices $w,w'\in\{w_1, \ldots, w_k\}$. Furthermore, if there is a pair of consecutive parallel edges connecting a black vertex to a white vertex $w$, we add to $Q_1$ a loop-arrow $w\to w$, see Figure \ref{figure:quiverExample}. Black vertices of degree 1 do not contribute any arrows to $Q_1$. 

Going forward, given a dessin $D$, we denote the multi-set of edges connecting two vertices $v$ and $v'$ of $D$, by $e(v, v')$.  

	\begin{figure}[ht]
		\centering
		\includegraphics[width=.5\textwidth]{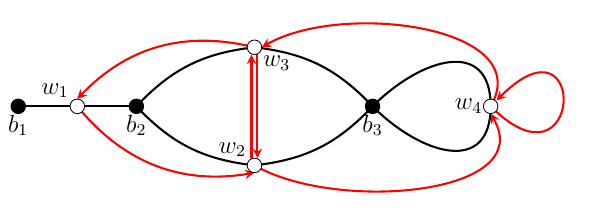}
		\caption{The quiver $Q_D$ of the dessin $D$ from Figure \ref{figure:clean} is shown in red. Note that the two parallel edges connecting $b_3$ and $w_4$ give rise to a loop-arrow.}\label{figure:quiverExample}
	\end{figure}

\begin{definition}\label{definition:relations} Let $b$ be a black vertex of degree $k\geq2$. A cycle of $k$ distinct arrows in $Q_D$ that are consecutive in the counter-clockwise ordering around $b$ is called a \emph{special} $b$\emph{-cycle}. Moreover, a special $b$-cycle starting at a white vertex  $w$ is called \emph{a special} $b$\emph{-cycle at} $w$.
\end{definition}

Given a field $K$ and a quiver $Q$, the \emph{path algebra} $KQ$ has vector space basis given by all possible finite paths in $Q$, including a trivial path $e_j$ for every vertex $j \in Q_0$. The  multiplication of two paths is given by concatenation if possible and zero otherwise. The multiplicative identity of $KQ$ is given by $\sum_{j \in Q_0} e_j$, the sum  of the trivial paths in $Q$. If $Q$ has an oriented cycle then $KQ$ is an infinite dimensional algebra. If $I$ is an ideal of $KQ$ satisfying certain admissibility condition, we call $KQ/I$ a \emph{bound quiver algebra}. For further details on paths algebras of quivers and bound quiver algebras, see for example \cite{Assem, Schiffler}.

Let $K$ be a field and $D$ be a clean dessin with $Q_D=(Q_0,Q_1)$. If $Q_1=\emptyset$, which is the case only when $D$ is the genus 0 dessin with two black vertices and one edge and with \belyi map $z\mapsto 4z(1-z)$, we by convention assign to it the quotient polynomial algebra $k[x]/(x^2)$. Otherwise, let $KQ_D$ be the path algebra of the quiver $Q_D$, and let $I_D=\left\langle\rho_D\right\rangle$ be the ideal of $KQ_D$ generated by the following relations.

\subsubsection*{Relations of type one.} For each white vertex $w$ and each pair $b_i$ and $b_j$ of black vertices of degree at least 2 adjacent to $w$, all relations of the form
\[C_i - C_j,\]
where $C_i$ and $C_j$ are special $b_i$ and $b_j$ cycles at $w$ are in $\rho_D$. Note that we allow $b_i=b_j=b$ in which case $C_i$ and $C_j$ range over all special $b$-cycles at $w$ (hence a non-trivial relation occurs only when $e(b,w)$ contains at least two parallel edges). 

\subsubsection*{Relations of type two.} For all black vertices $b$, all relations of the type $Ca$ are in $\rho_D$, where $C$ ranges across all special $b$-cycles, and $a$ is the first arrow of $C$.

\subsubsection*{Relations of type three.} All paths $a_ia_j$ of length 2 which are not subpaths of any special cycle are relations in $\rho_D$.

\begin{example}
	Consider the dessin $D$ from Figure \ref{figure:clean} together with its labeled quiver $Q_D$ shown in Figure \ref{figure:quiverAlgebraExample}.
	
	\begin{figure}[ht]
		\centering
		\includegraphics[width=.5\textwidth]{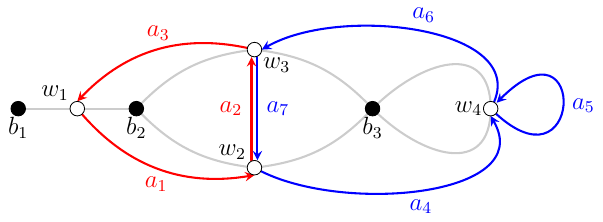}
		\caption{The labeled quiver $Q_D$ of the dessin $D$ from Figure \ref{figure:clean}. The special $b_2$-cycle given by rotations of $a_1a_2a_3$ is shown in red, and the  special $b_3$-cycle given by rotations $a_4a_5a_6a_7$ is shown in blue.}
		\label{figure:quiverAlgebraExample}
	\end{figure}

	There are three special $b_2$-cycles $a_1a_2a_3$, $a_2a_3a_1$ and $a_3a_1a_2$ at $w_1$, $w_2$ and $w_3$, respectively, and there are four special $b_3$-cycles $a_4a_5a_6a_7$, $a_7a_4a_5a_6$, $a_5a_6a_7a_4$ and $a_6a_7a_4a_5$. The first two are special $b_3$-cycles at $w_2$ and $w_3$, respectively, whereas the last two are both special $b_3$-cycles at $w_4$.
	
	The relations of type one are the following: the white vertex $w_1$ contributes no relations. The white vertex $w_2$ contributes $a_2a_3a_1-a_4a_5a_6a_7$. The white vertex $w_3$ contributes $a_3a_1a_2 - a_7a_4a_5a_6$. The white vertex $w_4$ contributes $a_5a_6a_7a_4-a_6a_7a_4a_5$. 
	
	Relations of type two are as follows: the black vertex $b_1$ contributes no relations. The black vertex $b_2$ contributes $a_1a_2a_3a_1$, $a_2a_3a_1a_2$ and $a_3a_1a_2a_3$, and the black vertex $b_3$ contributes $a_4a_5a_6a_7a_4$, $a_7a_4a_5a_6a_7$, $a_5a_6a_7a_4a_5$ and $a_6a_7a_4a_5a_6$.
	
	Finally, relations of type three are: $a_1a_4$, $a_2a_7$, $a_4a_6$, $a_5^2$, $a_6a_3$ and $a_7a_2$.
\end{example}

\begin{definition}
	The \emph{bound quiver algebra} $KQ_D/I_D$ is called a \emph{Brauer graph algebra}.
\end{definition}

In general, Brauer graph algebras are given by \emph{Brauer graphs}, i.e.\ tuples $G=(V,E,\mu,\mathfrak o)$ where $G=(V,E)$ is a graph with $V$ and $E$ as its vertex and edge-sets, together with a function $\mu\colon V\to\mathbb N$ called \emph{multiplicity} and an orientation $\mathfrak o$ given by the choice of a cyclic ordering of the edges in $E$ around the vertices in $V$. Hence we consider a clean dessin $D$ as the Brauer graph $G=(V,E,\mu,\mathfrak o)$ where $V$ is taken to be the set of black vertices of $D$, $E$ is given by the sets $\{b_i,b_j\}\subset V$ such that $b_i$ and $b_j$ are adjacent to the same white vertex of $D$, the multiplicity function maps all black vertices to $1$ and $\mathfrak o$ is the counter-clockwise orientation. 

Brauer graph algebras, see for example \cite{benson} or \cite{schrollbrauer}, coincide with the class of  symmetric special biserial algebras \cite{Schroll} when $K$ is algebraically closed.

\section{Galois conjugate Brauer graph algebras are derived equivalent}

In this section we show that the Galois conjugate of a  bipartite clean dessin is again bipartite.   Together with a recent result by Opper and Zvonareva \cite{Opp-Zvn} this shows that the Brauer graph algebras associated to Galois conjugate clean dessins are derived equivalent.

We start by showing that for a clean dessin the property of being bipartite is preserved by the Galois action. 

Note that a clean dessin $D$ is bipartite if it is bipartite when we consider it as a graph whose only vertices are the black vertices of $D$. That is, if $B$ is the set of black vertices of a clean dessin $D$, and $W$ is the set of its white vertices, define $E$ as the set of pairs of black vertices adjacent to the same white vertex from $W$. Then $D$ is bipartite as a clean dessin if the graph $(B,E)$ is bipartite.

\begin{theorem}\label{lemma:bipartite}
If $D$ is a bipartite clean dessin, then any Galois conjugate of $D$ is also bipartite.
\end{theorem}
\begin{proof}
Let $D=(X_D,f_D)$ be a bipartite clean dessin and let $(U,V)$ be a bipartition of the black vertices of $D$. Consider now a new dessin $G=(X_G,f_G)$ such that $X_G$ is isomorphic to $X_D$ as a topological surface, with black vertices given by $U$, white vertices given by $V$ and edges $E$ such that $e\in H\subseteq U\times V$ if $e=(u,v)$ is an edge of $D$ (see Figure \ref{figure:bipartite}). Note that in general the dessin $G$ will not be clean.

\begin{figure}[ht]
	\centering
	\includegraphics[width=.25\textwidth]{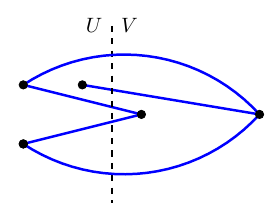}\includegraphics[width=.25\textwidth]{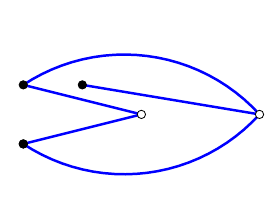}
	\caption{An example of a bipartition $(U,V)$ of a clean dessin and the corresponding dessin $G$. Note that $G$ is not clean.}\label{figure:bipartite}
\end{figure}

Let $g\colon\rsphere\to\rsphere$ be the function defined by $z\mapsto4z(1-z)$ for $z\in\mathbb C$ and $\infty\mapsto\infty$. Recall from Section \ref{section:clean} that if $(X,f)$ is a dessin, then $(X,g\circ f)$ is a clean dessin obtained from $(X,f)$ by colouring all the white vertices black and adjoining new white vertices between any two black vertices \cite{GGD}.

Consider the clean dessin $(X_G,g\circ f_G)$ obtained from $G$. This dessin is clearly isomorphic (as a cellularly embedded graph) to $D$, hence by Grothendieck's correspondence  $(X_D,f_D)$ and $(X_G,g\circ f_G)$ are isomorphic. 

Now if $\theta\in\absgal$, then the action of $\theta$ on $D$ results in the same dessin as the action of $\theta$ on $(X_G,g\circ f_G)$, that is
\[(X_D,f_D)^\theta=(X_G,g\circ f_G)^\theta=(X_G^\theta,(g\circ f_G)^\theta),\]
where all the equalities are up to isomorphisms of dessins. Because the coefficients of $g$ are rational, the action of $\theta$ on $g$ is trivial and we have $(g\circ f_G)^\theta=g\circ f_G^\theta$. It therefore follows that
\[(X_D,f_D)^\theta=(X_G^\theta,g\circ f_G^\theta).\]
Finally, the clean dessin $(X_G^\theta,g\circ f_G^\theta)$ has a bipartition of its black vertices obtained as follows: let $P$ and $Q$ be the sets of black and white vertices of $G^\theta=(X_G^\theta,f_G^\theta)$, respectively. Since the black vertices of the clean dessin $(X_G^\theta,g\circ f_G^\theta)$ are obtained by colouring the white vertices of $G^\theta$ in black, then the black vertices of $(X_G^\theta,g\circ f_G^\theta)$ have a bipartition $(P,Q)$. Therefore $(X_D,f_D)^\theta=(X_G,g\circ f_G)^\theta$ is a bipartite clean dessin.
\end{proof}
The property of not being bipartite is also preserved by the Galois action, and follows  from the Theorem.
\begin{corollary}
	If $D$ is a clean dessin which is not bipartite, then any Galois conjugate of $D$ will also not be bipartite.
\end{corollary}
\begin{proof}
	If $D$ is a clean dessin which is not bipartite with a Galois conjugate $D^\theta$ that is bipartite, then Lemma \ref{lemma:bipartite} would imply that $(D^\theta)^{\inv{\theta}}=D$ is bipartite, which is clearly a contradiction.
\end{proof}

From now on, we assume the field $K$ of a Brauer graph algebra to be algebraically closed. We now briefly recall the main result from \cite{Opp-Zvn}. 

\begin{theorem}\cite[Theorem 7.11]{Opp-Zvn}\label{thm:Opp-Zvn} Let $D_1$ and $D_2$ be clean dessins d'enfants such that their associated Brauer graph algebras $\Lambda_1$ and $\Lambda_2$ are not local. Then $\Lambda_1$ and $\Lambda_2$ are derived equivalent if and only if the following data for $D_1$ and $D_2$ coincides:
\begin{itemize}
	\item[i)] the number of black vertices and edges,
	\item[ii)] the number and the degrees of faces,
	\item[iii)] the multiset of multiplicities,
	\item[iv)] either both $D_1$ and $D_2$ are bipartite, or neither is.
\end{itemize} 
\end{theorem}

By Theorem \ref{theorem:galinv} the number of black and white vertices, edges and faces, as well as the degrees of the faces are all preserved by the Galois action, and moreover, the multisets of multipicities coincide by definition (i.e.\ the multiplicity function $\mu$ assigns multiplicity 1 to each black vertex).

Combining Theorem~\ref{lemma:bipartite} with Theorem~\ref{thm:Opp-Zvn} we obtain that two Galois conjugate dessins give rise to derived equivalent Brauer graph algebras. 

\begin{corollary}\label{corollary:derivedeq} Let $D_1$ and $D_2$ be Galois conjugate clean dessins d'enfants with Brauer graph algebras $\Lambda_1$ and $\Lambda_2$. Then $\Lambda_1$ and $\Lambda_2$ are derived equivalent.  
\end{corollary}

We note that the only local Brauer graph algebra arising from a dessin d'enfant is \sloppy $k[x,y]/(x^2,y^2)$ for which the corresponding dessin is a simple loop. This dessin has $z\mapsto \frac1{4z(z-1)}$ as its \belyi function and therefore is the only member of its Galois conjugacy class. So Corollary 3.4 still holds in the case of a local Brauer graph algebra arising from a dessin d'enfant.

The following result is an indication that Galois conjugation is more refined than the  classification by derived equivalence of the associated Brauer graph algebras. Namely, we show that two Galois conjugate Brauer graph algebras have the same dimension (and this is usually not preserved by derived equivalence). Therefore Galois orbits give a stratification  of the derived equivalence classes of Brauer graph algebras.

\begin{theorem}\label{theorem:dimension}Let $D$ be a clean dessin, $Q_D=(Q_0,Q_1)$ its quiver, and $\Lambda_D=KQ_D/I_D$ the associated Brauer graph algebra. The dimension $\dim_K\Lambda_D$ of $\Lambda_D$ is a Galois invariant.
\end{theorem}

\begin{proof}
In \cite{GreenSchrollBrauerConfig} it is shown that a basis of a Brauer graph algebra $\Lambda_D$ induced by a dessin $D$ is given by the subpaths of all special $b$-cycles modulo the admissible ideal, $\{p+I_D\mid p\textnormal{ is a subpath of a special $b$-cycle in }Q_D\}$,
including the trivial paths for each vertex of a dessin. Thus the dimension $\dim_K\Lambda_D$ is given by

\[\dim_K\Lambda_D=2|Q_0|+\sum_{b}|b|(|b|-1)=2|W|+\sum_{\{b\mid\deg b\geq 2\}}\deg b(\deg b-1).\]
Here $Q_0$ is the set of vertices of $Q$, $|b|$ denotes the length of any special $b$-cycle induced by the black vertex $b$, and the sum runs over all special $b$-cycles (up to permutation). Note that we pick only one special $b$-cycle per black vertex in the summation above. Equivalently, the dimension is given in terms of the set $W$ of white vertices of $D$ and the degrees of black vertices of $D$ of degree at least 2. All terms in the dimension formula are Galois invariants, therefore the dimension of a Brauer graph algebra is also a Galois invariant.
\end{proof}

\begin{example}
	The two clean dessins shown in Figure \ref{figure:notSameDim} are derived equivalent but not of the same dimension. That they are derived equivalent follows from the fact that they are related by a \emph{Kauer mutation move} \cite{Kauer}. However, the triangle dessin has three special $b$-cycles of length 2 hence the dimension of its Brauer graph algebra is $3+(2+2+2)=9$, whereas the other dessin has only one special $b$-cycle of length 4 hence the dimension of its Brauer graph algebra is $3+4\cdot 3=15$.
	
	\begin{figure}
		\centering
		\includegraphics[width=.75\textwidth]{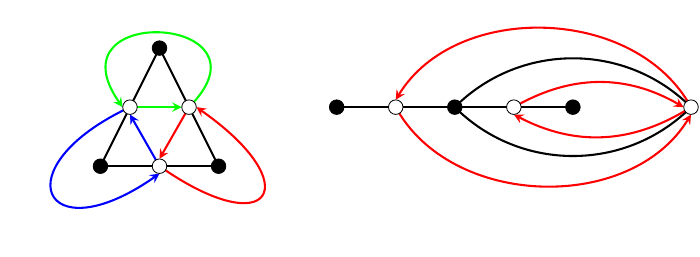}
		\caption{Two clean dessins whose Brauer graph algebras are derived equivalent but not of the same dimension. The clean dessin on the left has three special cycles each of length 2 (show in red, green and blue) and the clean dessin on the right has only one special cycle of length 4 (shown in red).}\label{figure:notSameDim}
	\end{figure}
\end{example}

The faces of a dessin $D$ are in 1-1 correspondence with the \emph{Green walks} around the Brauer graph of $A_D= KQ_D/I_D$. Green walks for Brauer graph algebras have been defined in \cite{Green, Roggenkamp}. 
Let $F$ be a face of a clean dessin of degree $k$ bounded by the edges $e_1$, $\dots$, $e_k$, where $e_{i}$ precedes $e_{i+1}$ in the counter-clockwise cyclic order around $F$ for $i=1,\dots,k-1$ and $e_k$ precedes $e_1$.
\emph{A Green walk starting at} $e_1$ is an infinite periodic sequence $(e_{i \mod{k}})_{i\in\mathbb N}$. The period of a Green walk starting at $e$ is equal to the degree of the face containing $e$.

Each face of odd degree $d$ gives rise to a connected component of the stable Auslander-Reiten quiver of $A_D$ which is an exceptional tube of rank $d$, that is a component of the form  $\mathbb{Z} A_\infty/\tau^{d}$ where $\tau$ is the Auslander-Reiten translate. Each face of even degree $2d$ gives rise to two connected components which are  exceptional tubes of rank $d$, see \cite{Duffield}.

\begin{proposition}\label{proposition:excptubes}Let $D$ be a dessin d'enfant and $A_D$ the associated Brauer graph algebra and suppose that $K$ is algebraically closed. Then the number and the rank of the exceptional tubes in the stable Auslander-Reiten quiver of $A_D$ are Galois invariants.
\end{proposition}
\begin{proof} By \cite{Duffield}, the exceptional tubes in the Auslander-Reiten quiver of $A_D$ are in bijection with the double-stepped Green walks on the Brauer graph. Suppose that  $D=\per$. Then the  Green walks are in bijection with the cycles of $\varphi$ and the double-stepped Green walks correspond to the cycles of $\varphi^2$. The result now follows from the fact that the cycle structure of both $\varphi$ and $\varphi^2$ are Galois invariants \cite{GGD, LandoZvonkin}.
\end{proof}

\subsection*{Statements and Declarations}
There is no data associated to the publication. 

\subsection*{Compliance with Ethical Standards}
The publication is compliant with ethical standards.

\end{document}